\pdfoutput=1 

\documentclass[11pt]{amsart}
\usepackage{amsmath,amsthm,amssymb,amscd,eucal}        
\usepackage{graphicx}
\usepackage[margin=1in]{geometry}

\newcommand{\OM} [1] {\ensuremath{{\overline{\mathcal M}}{_{0, #1}^{{\rm or}}(\R)}}}

\def\R{\mathbb{R}}

\newcommand{\T}[1]{{\mathcal T}_#1}                    % Tree space
\newcommand{\K}{{\mathcal K}}                          % Associahedron Dual
\newcommand{\Pe}{{\mathcal P}}                         % Permutohedron Dual

\newcommand{\St} {{\rm St}}                            % Star complex
\newcommand{\Sg} {{\mathbb S}}                         % Symmetric group
\newcommand{\BHV}[1]{{\rm BHV}_{#1}}                    % BHV space of trees

\newcommand{\hide}[1]{}

\theoremstyle{plain}
\newtheorem{thm}{Theorem}
\newtheorem{prop}[thm]{Proposition}
\newtheorem{cor}[thm]{Corollary}

\theoremstyle{definition}
\newtheorem*{defn}{Definition}
\newtheorem*{exmp}{Example}

\theoremstyle{remark}
\newtheorem*{rem}{Remark}
\newtheorem*{ack}{Acknowledgments}

\numberwithin{equation}{section}

\begin{document}

\title{Polyhedral Covers of Tree Space}

\author[Devadoss]{Satyan L.\ Devadoss}
\address{S.\ Devadoss: Williams College, Williamstown, MA 01267}
\email{satyan.devadoss@williams.edu}

\author[Huang]{Daoji Huang}
\address{Huang: Cornell University, Ithaca, NY, 14853}
\email{dh539@cornell.edu}

\author[Spadacene]{Dominic Spadacene}
\address{Spadacene: University of Michigan, Ann Arbor, MI 48109}
\email{domspad@umich.edu}

\begin{abstract}
The phylogenetic tree space, introduced by Billera, Holmes, and Vogtmann, is a cone over a simplicial complex.  In this short article, we construct this complex from local gluings of classical polytopes, the associahedron and the permutohedron.  Its homotopy is also reinterpreted and calculated based on polytope data.
\end{abstract}

%\subjclass[2000]{05C05, 52C25, 92B10}


\maketitle

\baselineskip=17pt

%%%%%%%%%%%%%%%%%%%%%%%%%%%%%%%%%%%%%%%%%%%%%%%%%%%%%%%%%%%%%%%%%%%%%%%%%%%%%%%%%%%%%%%%%
%                                                      
%                Introduction
%
%%%%%%%%%%%%%%%%%%%%%%%%%%%%%%%%%%%%%%%%%%%%%%%%%%%%%%%%%%%%%%%%%%%%%%%%%%%%%%%%%%%%%%%%%
\section{Introduction}  \label{s:intro}

A \emph{phylogenetic tree} is a tree for which each internal edge is assigned a nonnegative length, each internal vertex has degree at least three, and each leaf has a unique labeling.  A classical problem in computational biology is the construction of a phylogenetic tree from a sequence alignment of species.  Billera, Holmes, and Vogtmann \cite{bhv} constructed an elegant space $\BHV{n}$ of isometry classes of rooted metric trees with $n$ labeled leaves.

Each such tree specifies a point in the orthant  $[0, \infty)^{n-2}$, parametrized by the lengths of their internal edges, and thus defines coordinate patches for the space of such trees. The space $\BHV{n}$ is assembled by gluing $(2n-3)!!$ orthants, the number of different binary trees on $n$ leaves \cite{dh}.  Two orthants of $\BHV{n}$ share a wall if and only if their corresponding binary trees differ by a \emph{rotation}, a move which collapses an interior edge of a binary tree, and then expands the resulting degree-four vertex into a different binary tree.  Figure~\ref{f:bhv3}(a) shows $\BHV{3}$ consisting of three rays glued at the origin, where a move from one ray to another is a rotation of the underlying trees.

\begin{figure}[h]
\includegraphics[width=.9\textwidth]{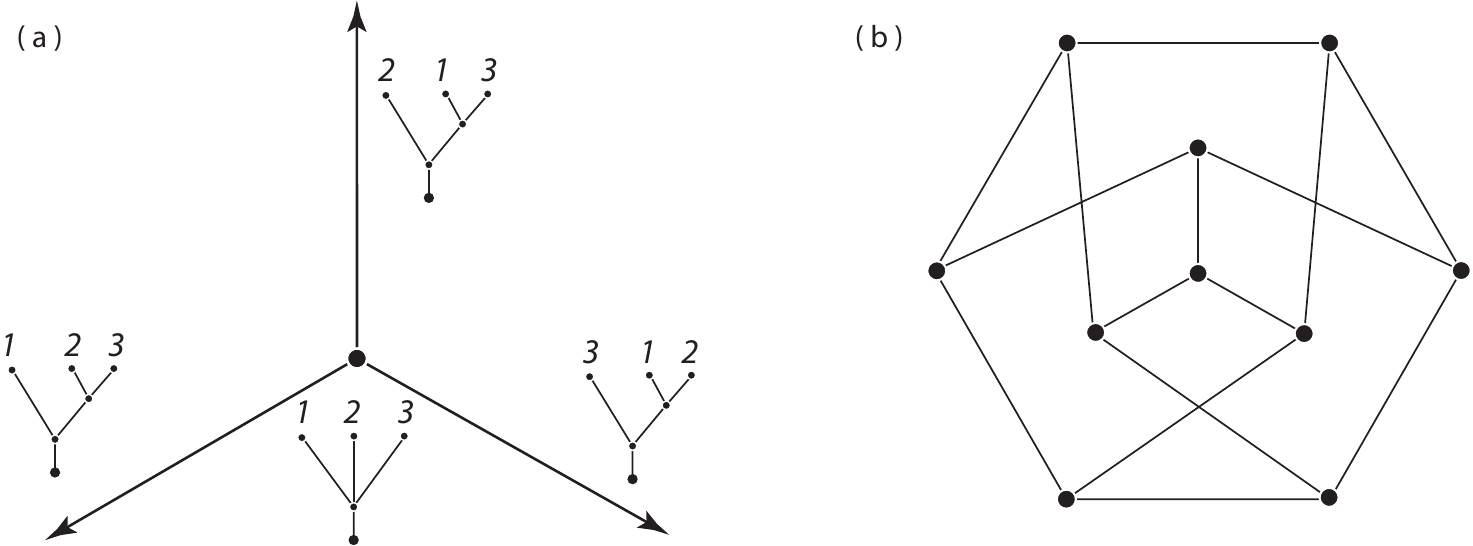}
\caption{(a) Tree space $\BHV{3}$ and (b) the simplicial complex $\T{4}$.}
\label{f:bhv3}
\end{figure}

Geometrically, this space has a CAT(0)-structure \cite{bhv}, enabling the computation of geodesics and centroids \cite{owe}.  Topologically, $\BHV{n}$ is contractible, and its one-point compactification is $\Sg_n$-equivariantly homotopy-equivalent to a version of the geometric realization of the poset of partitions of $n$ \cite[Section 9]{chi}, where the associated homology representations are fundamental in the theory of operads.  The structure of $\BHV{n}$ is captured by a space $\T{n}$, homeomorphic to Boardman's space of \emph{fully-grown} trees \cite{boa}.

\begin{defn}
Let $\T{n}$ be the subspace of $\BHV{n}$ consisting of trees with internal edge lengths that sum to 1.   It  is a pure simplicial $(n-3)$-complex composed of $(2n-3)!!$ chambers, with two adjacent chambers differing by a rotation of their underlying trees.  
In particular, $\T{n}$ has one $(k-1)$-simplex for every tree with $k$ interior edges.
\end{defn}

Indeed, $\BHV{n}$ is a cone over this space, where the cone-point is the degenerate tree with no internal edges.  For example, $\BHV{4}$ consists of 15 quadrants $[0, \infty)^2$ glued together, and its subspace $\T{4}$ is the Peterson graph with 15 edges, as displayed in Figure~\ref{f:bhv3}(b).  Here, the 10 vertices correspond to rooted trees with four leaves and one internal edge.  

These tree spaces $\T{n}$ have an importance of their own, from representation theory \cite{rw}, to moduli spaces \cite{dm}, to tropical geometry \cite{ss}.  This short article provides global descriptions of $\T{n}$ based on covering by classical polytopes that encapsulate algebraic information, notably the associahedron and the permutohedron.    Our construction of this complex from local gluings of simplices might provide new means of navigation in tree spaces, and their corresponding algorithms \cite{csj}.
Its homotopy, originally studied by Vogtmann \cite{vog}, and later by Robinson and Whitehouse \cite{rw}, is also reinterpreted in the polyhedral context.

\begin{rem}
There is a rich history between polytopes and tree spaces. Kapranov's permutoassociahedron \cite{kap}, a polytope blending the DNA of both associahedra and permutohedra, was initially considered a candidate for tree space itself \cite{hol}. More recently, the orientable double-cover of the real moduli space of curves \OM{n+1} is viewed as an alternative to $\T{n}$ \cite{dm}.  It is formed by gluing associahedra in a particular arrangement, one for each vertex of the permutohedron \cite{dev}.
\end{rem}

%%%%%%%%%%%%%%%%%%%%%%%%%%%%%%%%%%%%%%%%%%%%%%%%%%%%%%%%%%%%%%%%%%%%%%%%%%%%%%%%%%%%%%%%%
%                                                      
%                Associahedra
%
%%%%%%%%%%%%%%%%%%%%%%%%%%%%%%%%%%%%%%%%%%%%%%%%%%%%%%%%%%%%%%%%%%%%%%%%%%%%%%%%%%%%%%%%%
\section{Associahedra}  \label{s:assoc}

Let $A(n)$ be the poset of all bracketings of $n$ letters, ordered such that $a \prec a'$ if $a$ is obtained from $a'$ by adding new compatible brackets.   The \emph{associahedron} $K_n$ is a simple, convex polytope of dimension $n-2$ whose face poset is isomorphic to $A(n)$.  It appeared in the work of  Stasheff \cite{sta} in the 1960s, used in the homotopy theory of $H$-spaces.  Its vertices correspond to all different ways $n$ letters can be multiplied, each with a different associative grouping, and the famous \emph{Catalan numbers} enumerate them, with over 100 different combinatorial and geometric interpretations  available \cite{s2}.   Figure~\ref{f:2d3d}(a) shows the 2D associahedron $K_4$ with a labeling of its faces by bracketings, and part (b) shows $K_5$ with its nine facets.

\begin{figure}[h]
\includegraphics[width=.9\textwidth]{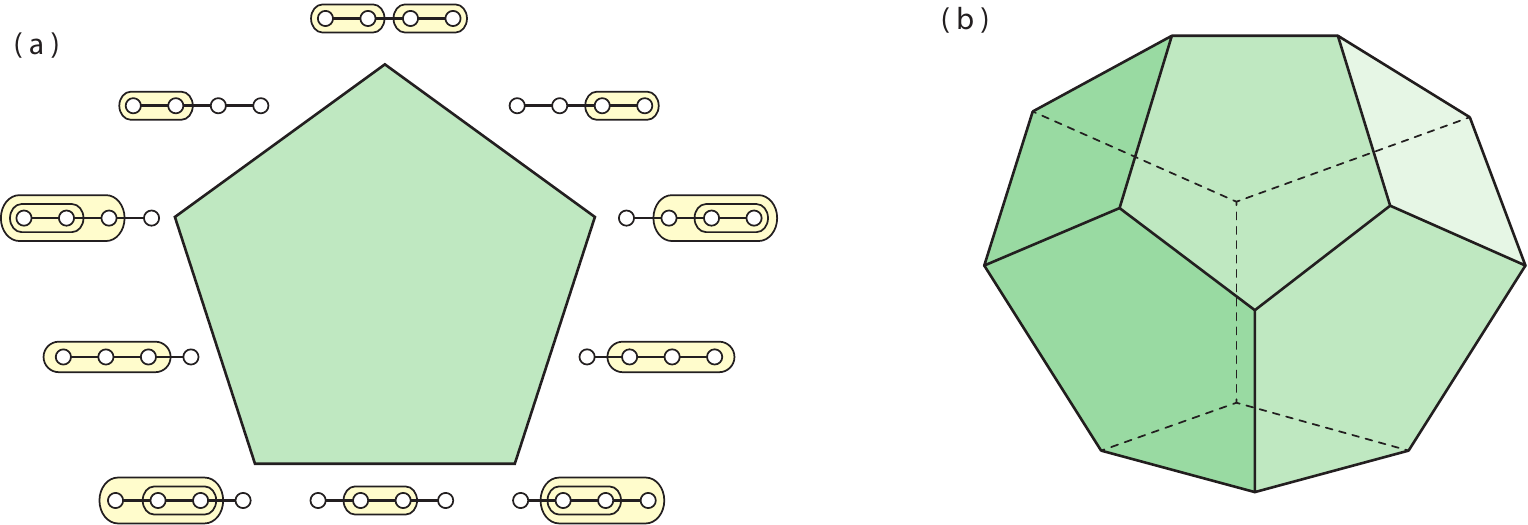}
\caption{Associahedra (a) $K_4$ from bracketings and (b) $K_5$.}
\label{f:2d3d}
\end{figure}

There is a natural bijection between bracketings on $n$ letters and planar rooted trees with $n$ leaves, labeled in a fixed order \cite{sta}.  We will be interested in the dual to the associahedron, described in the light of this relationship.   Since $K_n$ is a simple polytope, its dual is simplicial.

\begin{defn}
Let $\K_n$ be the boundary of the dual to $K_n$, the simplicial $(n-3)$-sphere whose $k$-simplices correspond to planar rooted trees with $n$ leaves and $k+1$ internal edges. 
\end{defn}

\noindent In particular, the chambers\footnote{A \emph{chamber} of a simplicial complex is a simplex of maximal dimension.} of $\K_n$ are identified with planar binary trees, where adjacent chambers differ by a rotation of their underlying trees.  Figure~\ref{f:2d3d-dual}(a) shows an example of $\K_4$ and the labeling of its five edges.  Part (b) shows $\K_5$, composed of 14 triangles, in bijection with the set of rooted binary trees with five leaves.  Compare with Figure~\ref{f:2d3d}.

\begin{figure}[h]
\includegraphics[width=.9\textwidth]{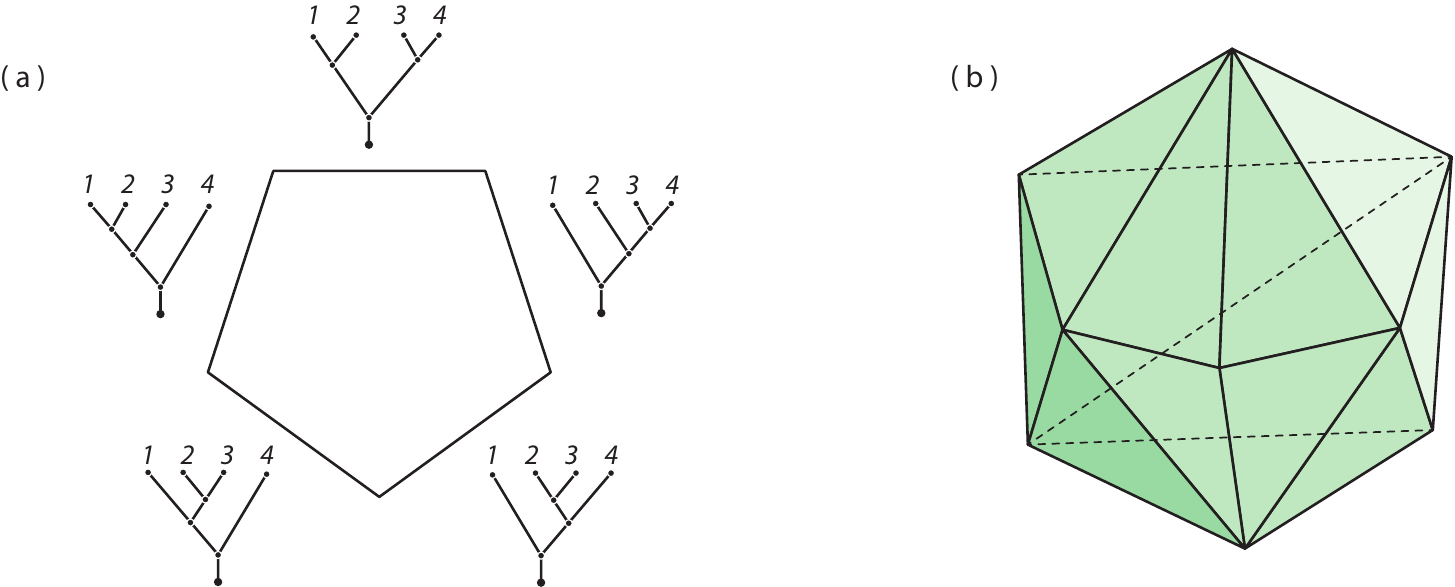}
\caption{The simplicial spheres (a) $\K_4$ and (b) $\K_5$.}
\label{f:2d3d-dual}
\end{figure}

\begin{prop} \label{p:assoc}
There are $n!/2$ distinct embeddings of $\K_n$ in $\T{n}$, with the symmetric group $\Sg_n$ acting on the labels.  Moreover, this set of $\K_n$ duals covers $\T{n}$, where each simplicial chamber of $\T{n}$ is contained in exactly $2^{n-2}$ distinct $\K_n$.
\end{prop}

\begin{proof}
Each $\K_n$ parameterizes the set of planar rooted trees with $n$ labeled leaves, in a fixed cyclic ordering.  Although there are $n!$ different labeling permutations, two labelings are identified up to order reversal since $\T{n}$ is not concerned with planarity.
A simplicial chamber of $\T{n}$ corresponds to a labeled rooted (nonplanar) binary tree.  Each internal edge creates a subtree (away from the root) which can be reflected, resulting in alternate planar embedding, and thus a new cyclic ordering, for the same tree.  As there are $n-2$ interior edges, there are $2^{n-2}$ duals that contain each chamber.
\end{proof}

\begin{exmp}
Figure~\ref{f:bhv-k4} shows the 12 different embeddings of $\K_4$ within the tree space $\T{4}$, where  each edge of $\T{4}$ is covered by exactly four distinct duals.
\end{exmp}

\begin{figure}[h]
\includegraphics[width=.95\textwidth]{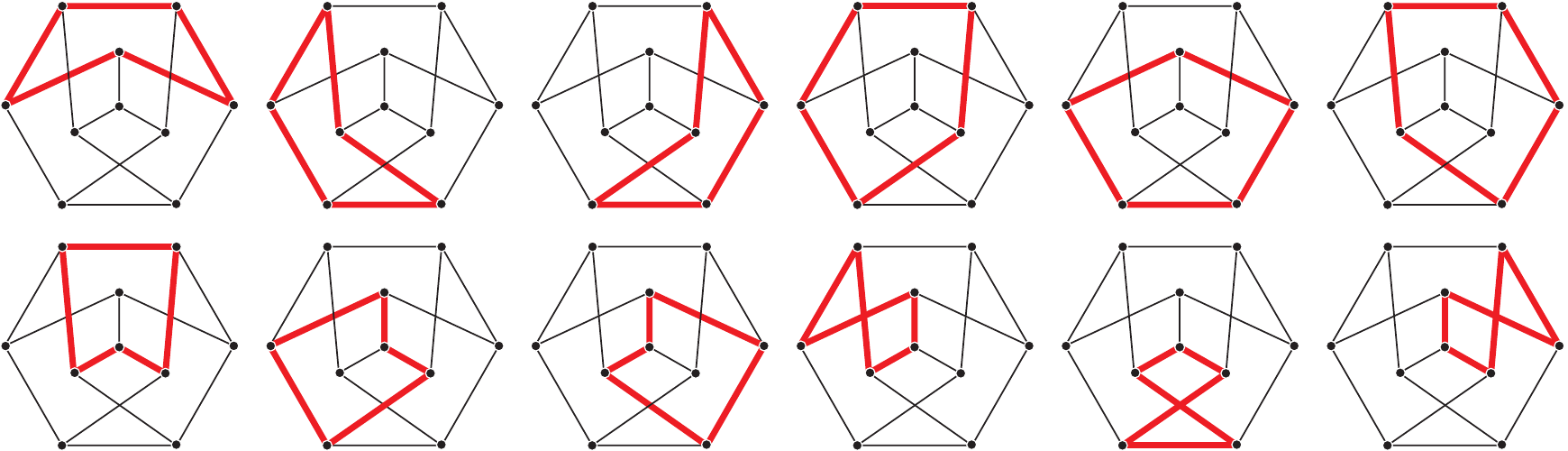}
\caption{There are 12 associahedra $\K_4$ duals in $\T{4}$.}
\label{f:bhv-k4}
\end{figure}

%%%%%%%%%%%%%%%%%%%%%%%%%%%%%%%%%%%%%%%%%%%%%%%%%%%%%%%%%%%%%%%%%%%%%%%%%%%%%%%%%%%%%%%%%
%                                                      
%                Permutohedra
%
%%%%%%%%%%%%%%%%%%%%%%%%%%%%%%%%%%%%%%%%%%%%%%%%%%%%%%%%%%%%%%%%%%%%%%%%%%%%%%%%%%%%%%%%%
\section{Permutohedra}  \label{s:permuto}

Let $B(n)$ be the poset of all order partitions on a set of $n$ letters, ordered such that $a \prec a'$ if $a$ is obtained from $a'$ by refining the partition.   The \emph{permutohedron} $P_n$ is a simple, convex polytope of dimension $n-1$ whose face poset is isomorphic to $B(n)$.  This classical object was studied by Schoute in the early twentieth-century, constructed as the convex hull of all vectors obtained by permuting the coordinates of $\langle 1, 2, \cdots, n \rangle$ in $\R^n$.    
\begin{figure}[b]
\includegraphics[width=.9\textwidth]{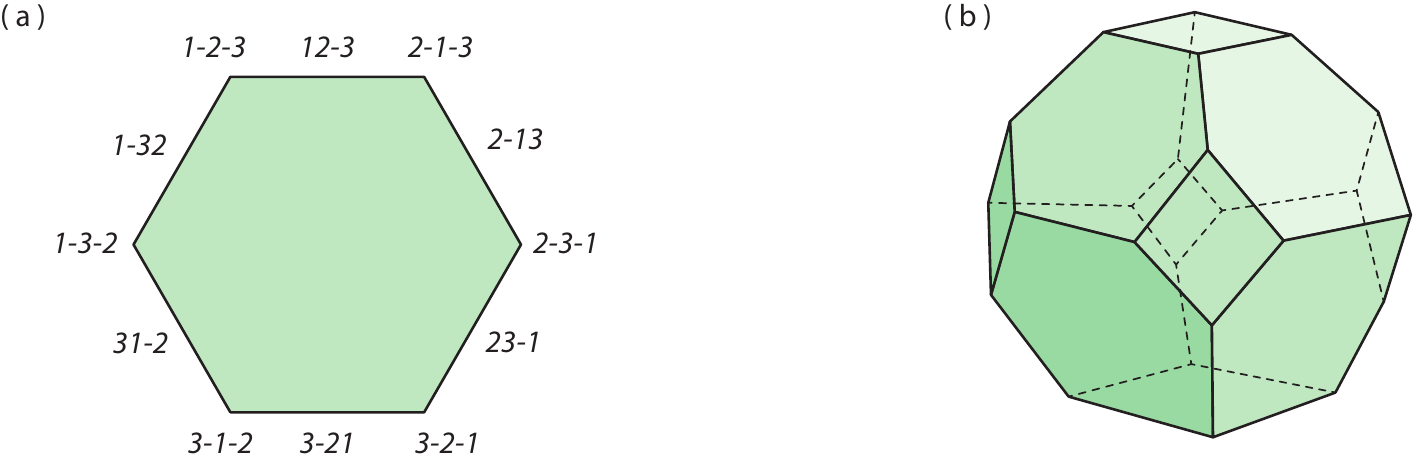}
\caption{Permutohedra (a) $P_3$ and (b) $P_4$.}
\label{f:permuto-2d3d}
\end{figure}
Indeed, as the associahedron captures associativity, the permutohedron encapsulates commutativity.  Figure~\ref{f:permuto-2d3d}(a) shows the 2D permutohedron $P_3$, whereas (b) displays the 3D version $P_4$.

A \emph{caterpillar} is a tree which becomes a path if all its leaves are deleted.\footnote{We assume, as always, that each internal vertex has degree at least three.}   There is a natural bijection between ordered partitions of $\{1, \dots, n\}$ and caterpillars with $n+2$ labeled leaves, with fixed labelings of $0$ and $n+1$ at either end of the caterpillar.  Figure~\ref{f:permuto-cats} shows examples for $n=5$; notice the number of elements in each partition matches the internal vertices of the caterpillar.

\begin{figure}[h]
\includegraphics[width=.95\textwidth]{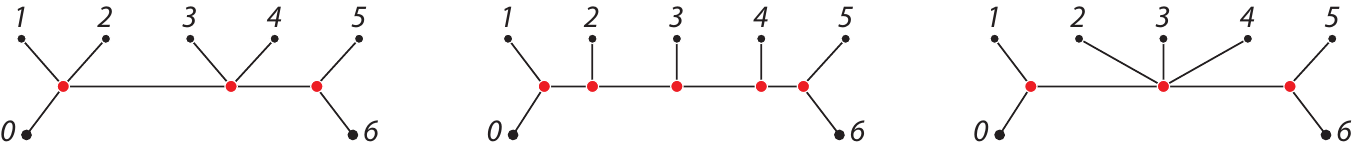}
\caption{Bijection between ordered partitions of $\{1, \dots, 5\}$ and labeled caterpillars.}
\label{f:permuto-cats}
\end{figure}

\begin{defn}
Let $\Pe_n$ be the boundary of the dual to $P_n$, a simplicial $(n-2)$-sphere, where each $k$-simplex corresponds to a caterpillar tree with $n+2$ leaves and $k+1$ internal edges.
\end{defn}

\noindent In particular, chambers of $\Pe_n$ are identified with binary caterpillars; the poset structure of the permutohedron reveals that two chambers are adjacent if their caterpillars differ by a rotation.

\begin{prop} \label{p:permuto}
There $\binom{n+1}{2}$ distinct embeddings of $\Pe_{n-1}$ in $\T{n}$.   Moreover, each simplicial chamber of $\T{n}$ that corresponds to a caterpillar is contained in exactly four distinct $\Pe_{n-1}$.
\end{prop}

\begin{proof}
Consider $\Pe_{n-1}$, viewed as caterpillar trees with $n+1$ labeled leaves.  This naturally embeds in $\T{n}$ by designating one of the leaves as the root.  Since we can choose two of the $n+1$ labels to be fixed at either end of the caterpillar, $\binom{n+1}{2}$ distinct embeddings exist.
For a simplicial chamber of $\T{n}$ with an underlying (binary) caterpillar, each end of this tree has exactly two leaves.  Choosing to fix a labeling for each pair results in four distinct $\Pe_{n-1}$ duals.
\end{proof}

\begin{exmp}
Figure~\ref{f:bhv-p4} shows the 10 different embeddings of $\Pe_3$ within the tree space $\T{4}$, where  each of its edges is covered by exactly four distinct duals.
\end{exmp}

\begin{figure}[h]
\includegraphics[width=.95\textwidth]{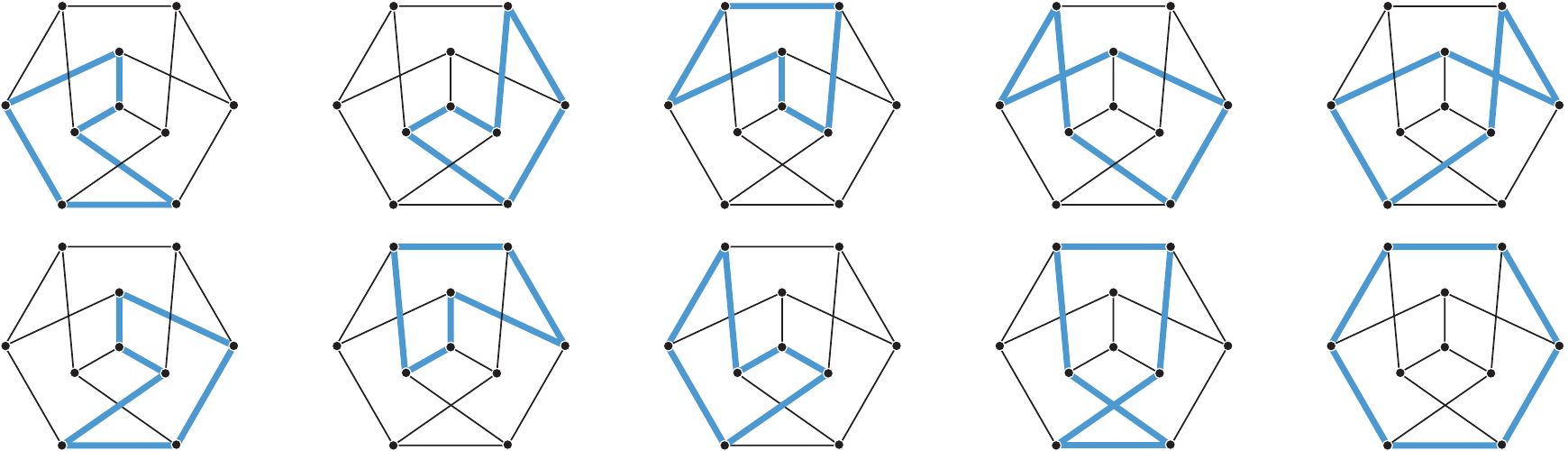}
\caption{There are 10 permutohedra $\Pe_3$ duals in $\T{4}$.}
\label{f:bhv-p4}
\end{figure}

\begin{rem}
Proposition~\ref{p:assoc} shows that the set of $\K_n$ embeddings covers $\T{n}$.  This is not true for the permutohedron version, since that only deals with caterpillars.  The case of $\T{4}$ is exceptional, as Figure~\ref{f:bhv-p4} shows, since all trees with five leaves are caterpillars.
\end{rem}

%%%%%%%%%%%%%%%%%%%%%%%%%%%%%%%%%%%%%%%%%%%%%%%%%%%%%%%%%%%%%%%%%%%%%%%%%%%%%%%%%%%%%%%%%
%                                                      
%                Bouquets
%
%%%%%%%%%%%%%%%%%%%%%%%%%%%%%%%%%%%%%%%%%%%%%%%%%%%%%%%%%%%%%%%%%%%%%%%%%%%%%%%%%%%%%%%%%
\section{Flowers and Bouquets}  \label{s:bouquets}

Throughout this section, without loss of generality, we let $\Pe_{n-1}$ denote the embedding in $\T{n}$ whose underlying caterpillars have labelings of $0$ and $n$ at its ends.   Thus, each chamber of $\Pe_{n-1}$ corresponds to a binary caterpillar with a unique permutation of the $n-1$ remaining leaves.  From Proposition~\ref{p:assoc}, each chamber also belongs to a distinct $\K_n$, by choosing leaf $0$ to be the root.

\begin{defn}
This collection of $(n-1)!$ distinct $\K_n$ duals, centered around $\Pe_{n-1}$, is called a \emph{flower}.  %Figure~\ref{f:flower}(a) shows the example for $n=4$.
\end{defn}

\begin{thm} \label{t:flower}
The embedding of the flower in $\T{n}$ covers it.
\end{thm}

\begin{proof}
Choose an arbitrary chamber in $\T{n}$, and choose a planar embedding $\tau$ of its underlying binary tree.  For a canonical representation, let all leaves lie on one side of the path from leaf $0$ to leaf $n$.  Figure~\ref{f:rotate}(a) shows a labeled tree, whereas part (b) gives a plane tree with leaves on one side of the path from $0$ to $9$.
\begin{figure}[h]
\includegraphics[width=.9\textwidth]{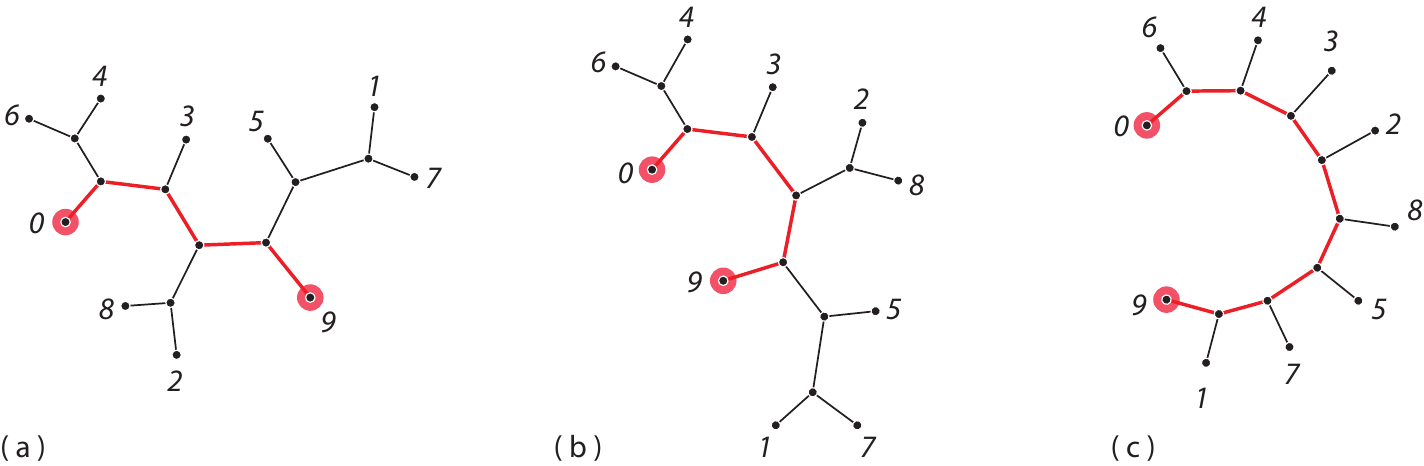}
\caption{(a) Labeled tree, (b) with leaves on one side, and (c) made into a caterpillar by rotations.}
\label{f:rotate}
\end{figure}
This associates $\tau$ to a particular embedding of $\K_n$ in $\T{n}$, call it $\K_n^\tau$,  based on the label order of $\tau$.  Now consider the set $E$ of interior edges of $\tau$ which do \emph{not} lie on the path between leaves $0$ and $n$.  Create a sequence of moves along the chambers of $\K_n^\tau$ by rotating edges of $E$, transforming $\tau$ into a caterpillar tree, and landing on a chamber of $\Pe_{n-1}$.  Figure~\ref{f:rotate}(c) displays the resulting tree based on rotations.
\end{proof}

\begin{cor}
There are $2^k$ distinct chambers of the flower that get identified in $\T{n}$, where $k$ is number of interior edges in its underlying tree $T$ not lying in the path between leaves $0$ and $n$.
\end{cor}

\begin{proof}
%A planar embedding of $T$ is needed such that all leaves lie on one side of the path from leaf $0$ to leaf $n$.  
For each internal edge not in the path, a different plane tree can be created by reflecting the subtree attached to the path along this edge, resulting in $2^k$ embeddings.  In particular, a chamber of the flower from the central $\Pe_{n-1}$ belongs to a unique $\K_n$, since $k=0$ for these caterpillars with labelings of $0$ and $n$ at its ends.
\end{proof}

\begin{cor}
All the $\K_n$ duals in the flower share the unique vertex $v_*$ in $\T{n}$, corresponding to the tree with one interior edge separating leaf labels $\{0,n\}$ from $\{1, \dots, n-1\}$.
\end{cor}

\begin{proof}
Each $\K_n$ in the flower is identified with a unique permutation of the $\{1, \dots, n-1\}$ leaf labels, and thus all meet at $v_*$.
\end{proof}

\begin{exmp}
Figure~\ref{f:flower}(a) shows a flower for the one-dimensional $n=4$ case.  It is comprised of six pentagons $\K_4$, identified around a central $\Pe_3$.  Part (b) shows the identifications when mapped to the tree space $\T{4}$.  Notice the unique marked vertex $v_*$, where all six associahedron duals meet in $\T{4}$.
The map from the flower (a) to the tree space (b) has the following structure:  there is a bijection between the six $\Pe_3$ edges of the flower and the corresponding (blue) edges of $\T{4}$.  Moreover, there is a 4:1 covering over each of the three (red) edges incident to $v_*$ in $\T{4}$, and a 2:1 cover of the remaining six (red) edges.
\end{exmp}

\begin{figure}[h]
\includegraphics[width=.9\textwidth]{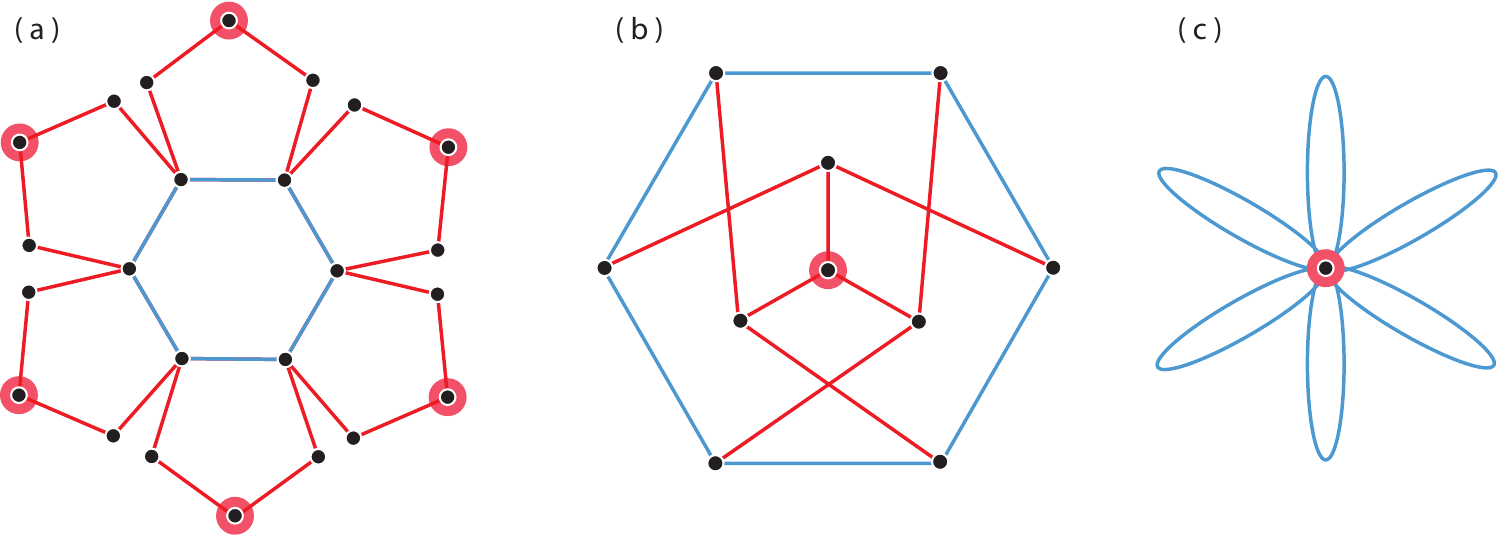}
\caption{(a) A flower around $\Pe_3$ and (b) its embedding in $\T{4}$, (c) which is homotopic to a bouquet of six circles.}
\label{f:flower}
\end{figure}

We close with a reinterpretation of a homotopy result of Vogtmann \cite{vog}, using the proof structure of Robinson and Whitehouse \cite{rw}.

\begin{defn}
Let $\St(v)$ denote the \emph{closed star} of vertex $v \in \T{n}$, the union of simplices of $\T{n}$ containing $v$. 
Let $v_{ij}$ be the vertex of $\T{n}$ whose underlying tree has one interior edge separating leaves $\{i,j\}$ from the remaining $n-1$ labels, such that $i,j \notin \{0,n\}$.
\end{defn}

\begin{thm}
The tree space $\T{n}$ is homotopic to a bouquet of $(n-1)!$ spheres of dimension $n-3$.  In particular, each chamber of $\Pe_{n-1}$ becomes a sphere as the rest of the flower contracts to a point. 
\end{thm}

\begin{proof}
Each  $\St(v_{ij})$ is naturally contractible and contains the vertex $v_*$.  Moreover, any intersection of a subcollection of stars $\{\St(v_{ij})\}$ is also contractible, since it is a conical subset of $\St(v_*)$.  Therefore, the union $\mathcal S$ of all the stars $\{\St(v_{ij})\}$ is contractible.

The complement of $\mathcal S$ in $\T{n}$ consists of exactly the interiors of the $(n-1)!$ chambers of $\Pe_{n-1}$, each corresponding to a binary caterpillar have labelings of $0$ and $n$ at its ends.  Since $\mathcal S$ is contractible, what remains of $\T{n}$ is a bouquet of $(n-3)$-spheres, one for each chamber of $\Pe_{n-1}$.  Figure~\ref{f:flower}(c) shows $\T{4}$ to be homotopic to a bouquet of six circles.
\end{proof}

\begin{ack}
Thanks go to Tom Nye for motivation, along with Craig Corsi, Susan Holmes and Karen Vogtmann for helpful conversations.  We  are grateful to Williams and the NSF for partial support with grant DMS-0850577, and to Gunnar Carlsson and Stanford for hosting Devadoss on his sabbatical.
\end{ack}

%%%%%%%%%%%%%%%%%%%%%%%%%%%%%%%%%%%%%%%%%%%%%%%%%%%%%%%%%%%%%%%%%%%%%%%%%%%%%%%%%%%%%%%%%
%
%                  REFERENCES
%
%%%%%%%%%%%%%%%%%%%%%%%%%%%%%%%%%%%%%%%%%%%%%%%%%%%%%%%%%%%%%%%%%%%%%%%%%%%%%%%%%%%%%%%%%

\bibliographystyle{amsplain}

\begin{thebibliography}{XXX}
\baselineskip=15pt

\bibitem{bhv} L.\ Billera, S.\ Holmes, K.\ Vogtmann. Geometry of the space of phylogenetic trees, \emph{Advances in Applied Mathematics} {\bf 27} (2001) 733--767.

\bibitem{boa} J.\ Boardman. Homotopy structures and the language of trees, in \emph{Proceeding of Symposia in Pure Mathematics} {\bf 21} (1971) 37--58.

\bibitem{chi} M.\ Ching. Bar constructions for topological operads and the Goodwillie derivatives of the identity, \emph{Geometry and Topology} {\bf 9} (2005) 833--933.

\bibitem{csj} S.\ Cleary and K.\ St.John.  Linear-time Approximation Algorithm for Rotation Distance, \emph{Journal of Graph Algorithms and Applications} {\bf 14} (2010) 385--390.

\bibitem{dev} S.\ Devadoss. Tessellations of moduli spaces and the mosaic operad, in \emph{Homotopy Invariant Algebraic Structures}, Contemporary Mathematics {\bf 239} (1999) 91--114.

\bibitem{dm} S.\ Devadoss and J.\ Morava.  Navigation in tree space, preprint {\tt arxiv:1009.3224}.

\bibitem{dh} P.\ Diaconis and S.\ Holmes. Matchings and phylogenetic trees, \emph{Proceedings
of the National Academy of Sciences} {\bf 95} (1998) 14600--14602.

\bibitem{hol} S.\ Holmes.  Personal communication, 2012.

\bibitem{kap} M.\ Kapranov. The permutoassociahedron, MacLane's coherence theorem, and asymptotic zones for the $KZ$ equation, {\em Journal of Pure and Applied Algebra} {\bf 85} (1993) 119--142.

\bibitem{owe} M.\ Owen. Computing Geodesic Distances in Tree Space, \emph{SIAM Journal on Discrete Mathematics} {\bf 25} (2011) 1506--1529.

\bibitem{rw} A.\ Robinson and S.\ Whitehouse. The tree representation of $\Sg_{n+1}$, \emph{Journal of Pure and Applied  Algebra} {\bf 111} (1996) 245--253.

\bibitem{ss} D.\ Speyer and B.\ Sturmfels. The tropical Grassmannian, \emph{Advances in Geometry} {\bf 4} (2004) 389--411.

\bibitem{s2} R.\ Stanley. \emph{Enumerative Combinatorics, Volume 2}, Cambridge University Press, 1999.

\bibitem{sta} J.\ Stasheff. Homotopy associativity of $H$-spaces, \emph{Transactions of the
American Mathematical Society} {\bf 108} (1963) 275--292.

\bibitem{vog} K.\ Vogtmann. Local structures of some out($F_n$)-complexes, \emph{Proceedings of the Edinburgh Mathematical Society} {\bf 33} (1990) 367--379.
\end{thebibliography}

\end{document}